\documentclass[reqno,a4paper]{amsart}

\usepackage[utf8]{inputenc}
\usepackage[english]{babel}
\usepackage{amssymb,mathrsfs,amsfonts,amsmath,amsthm}
\usepackage{graphicx}
\usepackage{xcolor}
\usepackage[all]{xy}

\newtheorem{theorem}{Theorem}[section]

\newtheorem{definition}[theorem]{Definition}

\newtheorem{remark}[theorem]{Remark}
\newtheorem{question}[theorem]{Question}

\begin{document}

  \title[The Baire classification of strongly separately continuous functions]%
{The Baire classification of strongly separately continuous functions on $\ell_\infty$}
\author[Olena Karlova \and Tom\'{a}\v{s} Visnyai]%
{Olena Karlova* \and Tom\'{a}\v{s} Visnyai**}

\newcommand{\acr}{\newline\indent}

\address{\llap{*\,} Chernivtsi National University,\acr
                   Faculty of Mathematics and Informatics,\acr
                   Department of Mathematical Analysis,\acr
                   Kotsyubyns'koho 2, 58 012 Chernivtsi,\acr
                   Ukraine}
\email{maslenizza.ua@gmail.com}

\address{\llap{**\,} Institute of Information Engineering, Automation and Mathematics, \acr
                   Faculty of Chemical and Food Technology, \acr
                   Slovak University of Technology in Bratislava,\acr
                   Radlinskeho 9, 812 37 Bratislava,\acr
                   Slovak Republic}
\email{tomas.visnyai@stuba.sk}


\maketitle

\begin{abstract}
We prove that for any $\alpha\in[0,\omega_1)$ there exists a strongly separately continuous function $f:\ell_\infty\to [0,1]$ such that $f$ belongs to the $(\alpha+1)$'th /$(\alpha+2)$'th/ Baire class and does not belong to the $\alpha$'th Baire class  if $\alpha$ is finite /infinite/.
\end{abstract}

\section{Introduction}\label{sec:intro}
The notion of real-valued strongly separately continuous  function defined on $\mathbb R^n$ was introduced and studied by  Dzagnidze in his paper \cite{Dzagnidze}. He proved that the class of all strongly separately continuous real-valued functions on $\mathbb R^n$ coincides with the class of all continuous functions. Later,  \v{C}in\v{c}ura,  \v{S}al\'{a}t and  Visnyai \cite{CSV} considered strongly separately continuous functions defined on the Hilbert space $\ell_2$ of sequences $x=(x_n)_{n=1}^\infty$ of real numbers such that $\sum\limits_{n=1}^\infty x_n^2<+\infty$ and showed that there are
essential differences between some properties of strongly separately continuous functions defined on $\ell_2$ and   the corresponding properties of functions on $\mathbb R^n$. In particular, they noticed that there exists a strongly separately continuous function $f:\ell_2\to \mathbb R$ which does not belong to the first Baire class. Extending these results, Visnya~\cite{TV} constructed a strongly separately continuous function $f:\ell_2\to \mathbb R$ of the third Baire class which is   not quasi-continuous at every point of $\ell_2$. It was shown recently in~\cite{KV:JMAA} that   for every $2\le \alpha<\omega$ there exists a strongly separately continuous function $f:\ell_p\to\mathbb R$ which belongs the $\alpha$'th Baire class and does not belong to the $\beta$'th Baire class on $\ell_p$ for $\beta<\alpha$, where $p\in[1,+\infty)$.

The aim of this paper is to generalize results from~\cite{KV:JMAA}  to the case of $p=+\infty$. We develop arguments from~\cite{EHS} and prove that for any $\alpha\in[0,\omega_1)$ there exists a strongly separately continuous function $f:\ell_\infty\to [0,1]$ such that $f$ belongs to the $(\alpha+1)$'th /$(\alpha+2)$'th/ Baire class and does not belong to the $\alpha$'th Baire class  if $\alpha$ is finite /infinite/.

\section{Definitions and notations}
Let $\ell_\infty$ be the Banach space of all bounded sequences of reals with the norm
$$\|x\|_\infty=\sup_{k\in\omega}|x_k|$$ for all $x=(x_k)_{k\in \omega}\in\ell_\infty$. For  $x,y\in\ell_\infty$ we denote
$ d_\infty(x,y)=\|x-y\|_\infty$.
 If $x\in \ell_\infty$  and $\delta>0$, then
$$
B_\infty(x,\delta)=\{y\in\ell_\infty:\|x-y\|_\infty<\delta\}..
$$

\begin{definition}\label{df:df01}{\rm
Let   $x^0=(x_k^0)_{k\in\omega}\in \ell_\infty$ and $(Y,|\cdot-\cdot|)$ be a metric space. A function $f:\ell_\infty\to Y$ is said to be
 \emph{strongly separately continuous at  $x^0$  with respect to the $k$-th variable} if
\begin{gather}\label{gath:cond_ssc}
 \forall \varepsilon>0 \quad\exists \delta>0 \quad\forall  x=(x_k)_{k\in\omega}\in B_\infty(x^0,\delta)\nonumber\\
|f(x_1,\dots,x_{k},\dots)-f(x_1,\dots,x_{k-1},x_k^0,x_{k+1},\dots)|< \varepsilon.
\end{gather}
 If ${f}$ is strongly separately continuous at $x^0$ with respect to each variable, then ${f}$ is said to be  \emph{strongly separately continuous at $x^0$}. Moreover,  ${f}$ is \emph{strongly separately continuous on $\ell_\infty$} if it is strongly separately continuous at each point of $\ell_\infty$.}
\end{definition}

Strongly separately continuous functions we will also call {\it ssc functions} for short.

\begin{definition}{\rm
A subset $A\subseteq X$ of a Cartesian product $X=\prod_{k=1}^\infty X_k$ of sets $X_1,X_2,\dots$ is called {\it $\mathcal S$-open}~\cite{KaRAEX:2015}, if
\begin{gather*}
 \sigma_1(a)=\{(x_k)_{k=1}^\infty\in X: |\{k:x_k\ne a_k\}|\le 1\}\subseteq A
\end{gather*}
for all $a=(a_k)_{k=1}^\infty\in A$.}
\end{definition}

If $x\in\ell_\infty$ and $N\subseteq\omega$, then we put
$$
\pi_N(x)=(x_k)_{k\in N}.
$$
In the case $N=\{n\}$, we write $\pi_n(x)$ instead of $\pi_{\{n\}}(x)$.

\section{Main result}

Define a function $(\alpha)^{\bullet}$ as the following
\begin{gather}
  (\alpha)^{\bullet}=\left\{\begin{array}{ll}
                              \alpha, & \alpha\in[0,\omega), \\
                              \alpha+1, &\alpha\in[\omega,\omega_1).
                            \end{array}   \right.
\end{gather}

\begin{theorem}\label{thm:BaireClass}
  For any $\alpha\in [0,\omega_1)$ there exists a strongly separately continuous function $f:\ell_\infty\to [0,1]$ which belongs to the $(\alpha+1)^\bullet$'th Baire class and does not belong to the $\alpha$'th Baire class on $\ell_\infty$.
\end{theorem}

\begin{proof} We define inductively transfinite sequences $(A_\alpha)_{1\le\alpha<\omega_1}$ and $(B_\alpha)_{1\le\alpha<\omega_1}$ of subsets of $\ell_\infty$ in the following way. Put
\begin{gather*}
A_1=\{(x_n)_{n=1}^\infty\in\ell_\infty: \exists m\,\,\forall n\ge m\,\,\, x_n=0\}\quad\mbox{and}\quad B_1=\ell_\infty\setminus A_1.
\end{gather*}
Let $(T_n:n\in\omega)$ be a partition of $\omega$ onto infinite sets $T_n=\{t_{n0},t_{n1},\dots\}$, where $(t_{nm})_{m\in\omega}$ is a strictly increasing sequence of numbers $t_{nm}\in\omega$. We put
$$
\ell_\infty^{T_n}=\{(x_{t_{nm}})\in\ell_\infty:t_{nm}\in T_n\,\,\,\forall m\in\omega\}.
$$
For every $n\in\omega$ we denote by $A_1^n$ /$B_1^n$/ the copy of the set $A_1$ /$B_1$/, which is contained in the space $\ell_\infty^{T_n}$.
Assume that for some $\alpha> 1$ we have  already defined sequences $(A_\beta)_{1\le\beta<\alpha}$ and  $(B_\beta)_{1\le\beta<\alpha}$ (and their copies $(A_\beta^n)_{1\le\beta<\alpha}$ and  $(B_\beta^n)_{1\le\beta<\alpha}$ in $\ell_\infty^{T_n}$) of subsets of $\ell_\infty$. Now we put
\begin{gather*}
  A_\alpha=\left\{\begin{array}{ll}
                    \bigcup_{m=1}^\infty \bigcap_{n=m}^\infty \pi_{T_n}^{-1}(B_\beta^n), & \alpha=\beta+1, \\
                    \bigcup_{n=1}^\infty \pi_{T_n}^{-1}(A_{\beta_n}^n), & \alpha=\sup\beta_n,
                  \end{array}
  \right.
  \end{gather*}
  and
  \begin{gather*}
 B_\alpha=\ell_\infty\setminus  A_\alpha.
\end{gather*}

{\sc Claim 1.} {\it For every $\alpha\in[1,\omega_1)$ the following statements are true:
\begin{enumerate}
\item\label{claim1:1} the sets $A_\alpha$ and $B_\alpha$ are $\mathcal S$-open in $\ell_\infty$;

\item\label{claim1:2} for any $y=(y_n)_{n=1}^\infty\in\ell_\infty$ with $y_n\ne 0$ for all $n\in\omega$ we have
\begin{gather*}\label{gath:multiply}
  x=(x_n)_{n\in\omega}\in A_\alpha\,\,\Leftrightarrow\,\, z=(x_n\cdot y_n)_{n\in\omega}\in A_\alpha.
\end{gather*}
  \end{enumerate}}

{\it Proof of Claim 1.} {\it (\ref{claim1:1}).} Evidently, $A_1$ and $B_1$ are $\mathcal S$-open. Assume that for some $\alpha<\omega_1$ the claim is valid for all $\beta<\alpha$. Let $\alpha=\beta+1$ be an isolated ordinal. Take any $x\in A_\alpha$ and $y\in\sigma_1(x)$. Then  there exists $m\in\mathbb N$ such that $\pi_{T_n}(x)\in B_\beta^n$ for all $n\ge m$. Since $\pi_{T_n}(y)\in\sigma_1(\pi_{T_n}(x))$  and $B_\beta^n$ is $\mathcal S$-open, $\pi_{T_n}(y)\in B_\beta^n$. Therefore, $y\in A_\alpha$. We argue similarly in the case where $\alpha$ is a limit ordinal.

{\it (\ref{claim1:2}).} We fix  $y=(y_n)_{n=1}^\infty\in\ell_\infty$ such that $y_n\ne 0$ for all $n\in\mathbb N$. The statement is true for $\alpha=1$, since $A_1=\sigma(0)$. Assume that for some $\alpha<\omega_1$ the property is valid for all $\beta<\alpha$. Let $\alpha=\beta+1$ for some $\beta$. The inductive assumption implies that
\begin{gather*}
\begin{array}{ccc}
  x\in A_\alpha & \Longleftrightarrow  & \exists m\in\mathbb N\,\,\, \forall n\ge m\,\,\, \pi_{T_n}(x)\in B_\beta^n  \\
  \phantom{a} & \phantom{a} & \Updownarrow \\
  z\in A_\alpha & \Longleftrightarrow & \exists m\in\mathbb N\,\,\, \forall n\ge m\,\,\, \pi_{T_n}(z)\in B_\beta^n
\end{array}
\end{gather*}
We argue similarly in the case of limit $\alpha$. \hfill$\Box$

Consider the equivalent metric
$$
d(x,y)=\min\{d_\infty(x,y),1\}
$$
on the space $\ell_\infty$.

{\sc Claim 2.} {\it For every $\alpha\in [1,\omega_1)$ the following condition holds:
\begin{itemize}
  \item[$(*)$] for every set $C\subseteq (\ell_\infty,d)$ of the additive /multiplicative/ class $\alpha$ there exists a contracting mapping $f:(\ell_\infty,d)\to (\ell_\infty,d)$ with the Lipschitz constant $L=\frac 12$ such that
      \begin{gather}
        C=f^{-1}(A_\alpha) \quad /C=f^{-1}(B_\alpha)/,\\
        |\pi_n(f(x))|< 1\quad \forall x\in\ell_\infty\,\,\,\forall n\in\omega.
      \end{gather}
\end{itemize}}

{\it Proof of Claim 2.}  We will argue by the induction on $\alpha$. Let  $C$ be an arbitrary $F_\sigma$-subset of $(\ell_\infty,d)$. Then there exists an increasing sequence $(C_n)_{n\in\omega}$ of of closed subsets of $(\ell_\infty,d)$ such that $C=\bigcup_{n\in\omega} C_n$.
Consider a map $f:\ell_\infty\to\ell_\infty$, defined by the rule
$$
f(x)=\bigl(\tfrac 12 d(x,C_1),\dots,\tfrac{1}{2}d(x,C_n),\dots\bigr)
$$
for all $x\in\ell_\infty$.

We show that $C=f^{-1}(A_1)$. Take $x\in C$ and choose $m\in\omega$ such that $x\in C_n$ for all $n\ge m$. Then $d(x,C_n)=0$ and $\pi_n(f(x))=0$ for all $n\ge m$. Hence, $x$ belongs to the right-hand side of the equality. Now we prove the inverse inclusion. Let $x\in f^{-1}(A_1)$. Then there exists $m\in\omega$ such that $\pi_n(f(x))=0$ for all $n\ge m$. Consequently, $d(x,C_n)=0$ for all $n\ge m$. Since $C_n$ is closed, $x\in C_n$ for all $n\ge m$. Therefore, $x\in\bigcup_{n\in\omega}C_n=C$.

Since
$$
d(f(x),f(y))\le d_\infty(f(x),f(y))=\sup_{n\in\omega}|\tfrac 12 d(x,C_n)-\tfrac 12 d(y,C_n)|\le \tfrac 12 d(x,y)
$$
for all $x,y\in\ell_\infty$, the mapping $f$ is contracting with the Lipschitz constant $L=\frac 12$. Moreover,
$$
|\pi_n(f(x))|=\tfrac{1}{2}d(x,C_n)< 1
$$
for every $n\in\omega$.

Assume that for some $\alpha<\omega_1$ the condition $(*)$ is valid for all $\beta<\alpha$. Let $C\subseteq (\ell_\infty,d)$ be any set of the $\alpha$'th additive class. Take an increasing sequence of sets $C_n$ such that $C=\bigcup_{n\in\omega} C_n$, where  every $C_n$ belongs to the multiplicative class $\beta$ if $\alpha=\beta+1$,  and in the case $\alpha=\sup\beta_n$ we can assume that $C_n$ belongs to the additive class $\beta_n$ for every $n\in\omega$.  By the inductive assumption there exists a sequence $(f_n)_{n\in\omega}$ of contracting maps $f_n:(\ell_\infty,d)\to (\ell_\infty,d)$ with the Lipschitz constant $L=\frac 12$ such that
\begin{gather}
  C_n=\left\{\begin{array}{ll}
                f_n^{-1}(B_\beta),  & \alpha=\beta+1,\\
                f_n^{-1}(A_{\beta_n}), & \alpha=\sup\beta_n,
             \end{array}\right.\\
  |\pi_m(f_n(x))|< 1\quad\forall x\in\ell_\infty\,\,\,\forall n,m\in\omega.
\end{gather}
For every $k\in\omega$ we choose a unique pair $(n(k),m(k))\in\omega^2$ such that
$$
k=t_{n(k)m(k)}\in T_{n(k)}.
$$
For all $x\in\ell_\infty$ and $n,m\in\omega$ we put $f_{nm}(x)=\pi_m(f_n(x))$ and consider a map $f:\ell_\infty\to\ell_\infty$, defined by the rule
$$
f(x)=\bigl(\tfrac{1}{2} f_{n(1)m(1)}(x),\dots,\tfrac{1}{2} f_{n(k)m(k)}(x),\dots \bigr)
$$
for all $x\in\ell_\infty$.  The inequalities
\begin{gather*}
 |f_{nm}(x)-f_{nm}(y)|=|\pi_m(f_n(x))-\pi_m(f_n(y))|\le\\
 \le \sup_{m\in\omega}|\pi_m(f_n(x))-\pi_m(f_n(y))|=d_\infty(f_n(x),f_n(y))
  \end{gather*}
  and
\begin{gather*}
 |f_{nm}(x)-f_{nm}(y)|\le 2
\end{gather*}
imply that
\begin{gather*}
\frac{1}{2}|f_{nm}(x)-f_{nm}(y)|\le d(f_n(x),f_n(y))\le \frac 12 d(x,y)
\end{gather*}
for all $x,y\in\ell_\infty$ and $n,m\in\omega$. Then
\begin{gather*}
  d(f(x),f(y))\le d_\infty(f(x),f(y))=\\
  =\sup_{k\in\omega}|\tfrac 12 (f_{n(k)m(k)}(x)-f_{n(k)m(k)}(y))|\le \frac 12 d(x,y)
\end{gather*}
for all $x,y\in\ell_\infty$. Therefore, $f:(\ell_\infty,d)\to (\ell_\infty,d)$ is a Lipshitz map with the constant $L=\frac 12$.

It remains to show that $C=f^{-1}(A_\alpha)$. Assume that  $\alpha=\beta+1$ (we argue similarly if $\alpha$ is limit). Let us observe that
$x\in C$ if and only if there exists $m\in\omega$ such that $f_n(x)\in B_\beta$ for all $n\ge m$. Since
$$
\pi_{T_n}(f(x))=\Bigl(\frac {1}{2} \pi_k(f_n(x))\Bigr)_{k\in T_n},
$$
we have
\begin{gather*}
  f_n(x)\in B_\beta \,\,\Longleftrightarrow\,\, \pi_{T_n}(f(x))\in B_\beta^n.
\end{gather*}
by statement (\ref{claim1:2}) of Claim 1. Therefore, $C=f^{-1}(A_\alpha)$.
\hfill$\Box$

{\sc Claim 3.} {\it For every $\alpha\in [1,\omega_1)$ the set $A_\alpha$  belongs to the additive class $\alpha$  and does not belong to the multiplicative class $\alpha$ in $\ell_\infty$.}

{\it Proof of  Claim 3.} If $\alpha=1$, then
$$
A_1=\bigcup_{n\in\omega} \{x\in\ell_\infty: |\{k\in\omega:x_k\ne 0\}|\le n\}
$$
is an $F_\sigma$-subset of $\ell_\infty$, since every set $\{x\in\ell_\infty: |k\in\omega:x_k\ne 0|\le n\}$ is closed. Consequently, $B_1$ is $G_\delta$-subset of $\ell_\infty$. Suppose that for some $\alpha\ge 1$ the set $A_\beta$ /$B_\beta$/ belongs to the additive /multiplicative/ class $\beta$ in $\ell_\infty$ for every $\beta<\alpha$. Since every projection $\pi_{T_n}:\ell_\infty\to\ell_\infty^{T_n}$ is continuous, the set $A_\alpha$ belongs to the additive class $\alpha$ in $\ell_\infty$ and the set $B_\beta$ belongs to the multiplicative class $\alpha$ in $\ell_\infty$.

Fix $\alpha\in[1,\omega_1)$. In order to show that $A_\alpha$ does not belong to the $\alpha$'th multiplicative class we assume the contrary. Claim 2 implies that there exists a contraction  $f:(\ell_\infty,d)\to (\ell_\infty,d)$ such that $A_\alpha=f^{-1}(B_\alpha)$. By the Contraction Map Principle, there would be a fixed point for  $f$, which implies a contradiction. \hfill$\Box$

Now we are ready to construct a function $f$ from the statement of the theorem.  Let $\alpha\in[0,\omega_1)$ be fixed. If $\alpha=0$, then we put $A=c$, where $c$ is the subspace of $\ell_\infty$ consisting of all convergent sequences of real numbers.  If $\alpha>0$, then previous steps imply the existence of an $\mathcal S$-open set $A\subseteq \ell_\infty$ such that $A$ belongs to the $(\alpha)^\bullet$'th additive class and does not belong to the $(\alpha)^\bullet$'th multiplicative class. In any case for every $x\in \ell_\infty$ we put
\begin{gather*}
    f(x)=\Bigl\{\begin{array}{ll}
                              1, & x\in A, \\
                              0, & x\not\in A.
                            \end{array}
\end{gather*}

We prove that $f:\ell_\infty\to[0,1]$ is strongly separately continuous. Fix $\varepsilon>0$, $k\in\omega$ and $x=(x_n)_{n\in\omega}\in\ell_\infty$. We put $\delta=1$ and notice that for all $y\in B_\infty(x,\delta)$ we have
\begin{gather*}
  y=(y_1,\dots,y_k,\dots)\in A \,\,\Longleftrightarrow \,\, z=(y_1,\dots,y_{k-1},x_k,y_{k+1},\dots)\in A,
\end{gather*}
since $A$ is $\mathcal S$-open. Therefore,
\begin{gather*}
  |f(y)-f(z)|=0
\end{gather*}
for all $y\in B_\infty(x,\delta)$ and $z=(y_1,\dots,y_{k-1},x_k,y_{k+1},\dots)$. Hence, $f$ is strongly separately continuous at $x$ with respect to the $k$'th variable.

Notice that both $A$ and $X\setminus A$ are of the $(\alpha+1)^\bullet$'th additive class, that is, $A$ is ambiguous set of the $(\alpha+1)^{\bullet}$'th class in $\ell_\infty$. It is well-known  that the characteristic function of any ambiguous set of the class $\xi$ in any metric space belongs to the $\xi$'th  Baire class~\cite[\S 31]{Ku1} for any $\xi\in[1,\omega_1)$. Therefore, $f\in {\rm B}_{(\alpha+1)^{\bullet}}(\ell_\infty,[0,1])$.

If $\alpha=0$, then $f$ is discontinuous exactly on $A$ and hence $f\not\in {\rm B}_0(\ell_\infty,[0,1])$.

In case $\alpha>0$ we assume that $f\in{\rm B}_\alpha(\ell_\infty,[0,1])$. Then $f$ belongs to the $(\alpha)^\bullet$'th Borel class. Therefore, $A=f^{-1}(1)$ is the set of the $(\alpha)^\bullet$'th multiplicative class in   $\ell_\infty$, which contradicts to the choice of $A$.
\end{proof}

\begin{remark}{\rm
The existence of  an ssc function $f:\ell_\infty\to [0,1]$ which is not Baire measurable  was proved in \cite{KV:PIGC:2017}. The Baire classification of ssc functions defined on $\mathbb R^\omega$ was studied in~\cite{KaRAEX:2015}.}
\end{remark}

Theorem~\ref{thm:BaireClass} suggests the following question.
\begin{question}
  Does there exist  a strongly separately continuous function $f:\ell_\infty\to[0,1]$ such that $f\in{\rm B}_{\omega+1}\setminus {\rm B}_{\omega}$?
\end{question}


\begin{thebibliography}{99}
\bibitem{CSV} J.~\v{C}in\v{c}ura, T.~\v{S}al\'{a}t and T.~Visnyai, {\it On separately continuous functions $f:\ell^2\to\mathbb R$}, Acta Acad. Paedagog. Agriensis, XXXI (2004), 11--18.

\bibitem{Dzagnidze} O. Dzagnidze, {\it Separately continuous function in a new sense are continuous}, Real Anal. Exchange, {\bf 24} (1998-99), 695--702.

\bibitem{EHS} R. Engelking, W. Holszty\'{n}ski, R. Sikorski, {\it Some examples of Borel sets}, Colloq. Math., 15 (1966), 271--274.

\bibitem{KaRAEX:2015} O. Karlova, {\it On Baire classification of strongly separately continuous functions}, Real Anal. Exch., {\bf 40} (1) (2014/2015), 1--11.

\bibitem{KV:PIGC:2017} O. Karlova, T. Visnyai, {\it Some remarks concerning strongly separately continuous functions on spaces $\ell_p$ with $p\in[1,+\infty]$}, Proc. Int. Geom. Center, to appear.

\bibitem{KV:JMAA}	O. Karlova, T. Visnyai, {\it On strongly separately continuous functions on sequence spaces}, J. Math. Analysis and Appl., {\bf 439} (1) (2016), 296--306.

\bibitem{Ku1} K. Kuratowski, {\it Topology  I}, Academic Press, 1966.

\bibitem{TV} T.~Visnyai, {\it Strongly separately continuous and separately quasicontinuous functions \mbox{$f:\ell^2\to\mathbb R$}}, Real Anal. Exchange, {\bf 38} (2) (2013), 499--510.
\end{thebibliography}
\end{document}